\documentclass[3p,number]{elsarticle}

\usepackage{mathtools}
\usepackage{amsmath,amsthm,amssymb}
\newtheorem{lemma}{Lemma}
\newtheorem{theorem}{Theorem}
\newtheorem{corollary}{Corollary}
\newtheorem{conjecture}{Conjecture}

\newcommand{\setof}[1]{\left\{{#1}\right\}}
\newcommand{\cset}[2]{\setof{#1\,:\,#2}}
\newcommand{\ivcc}[2]{\left[#1,#2\right]}
\newcommand{\ivco}[2]{\left[#1,#2\right)}
\newcommand{\ivoc}[2]{\left(#1,#2\right]}

\newcommand{\Nn}{{\mathbb N}}
\newcommand{\Rr}{{\mathbb R}}

\newcommand{\Snk}{S_{n, k}\,}

\newcommand{\Snkp}[1]{S^{#1}_{n, k}\,}
\newcommand{\Snkpm}{\Snkp{m}}

\newcommand{\diffk}{\Delta_{k}}
\newcommand{\diffkm}{\Delta_{k-1}}

\newcommand{\bspk}[1]{N_{{#1}, k}}
\newcommand{\bspkm}[2]{N_{{{#1}}, {k-#2}}}

\newcommand{\abs}[1]{\left|#1\right|}
\newcommand{\norm}[1]{\left\Vert#1\right\Vert}
\newcommand{\normi}[1]{\left\Vert#1\right\Vert_\infty}
\newcommand{\normop}[1]{\left\Vert#1\right\Vert_{op}}

\newcommand{\spacecf}{C(\ivcc{0}{1})}
\newcommand{\spacedf}[1]{C^{{#1}}(\ivcc{0}{1})}
\newcommand{\spacepp}{\mathcal{S}(n, k)}
\newcommand{\spaceppk}[1]{\mathcal{S}(n, {#1})}

\journal{Journal of Approximation Theory}

\usepackage{fancyhdr}
\renewcommand{\headrulewidth}{0.0pt}

\begin{document}

\begin{frontmatter}
  \title{A lower bound for the uniform Schoenberg operator}

  \author[up]{Johannes~Nagler}
  \ead{johannes.nagler@uni-passau.de}

  \author[ua]{Uwe~K\"ahler}
  \ead{ukaehler@ua.pt}

  \address[up]{Fakult\"at f\"ur Informatik und Mathematik, Universit\"at Passau, Germany}
  \address[ua]{CIDMA -- Center for R\&D in Mathematics and Applications, Universidade de Aveiro, Portugal}

  \begin{abstract}
    We present an estimate for the lower bound for the Schoenberg operator with equidistant knots
    in terms of the second order modulus of smoothness. We investigate the behaviour of iterates
    of the Schoenberg operator and in addition, we show an upper bound of the second order derivative of these iterates.
    Finally, we prove the equivalence between the approximation error and the second order modulus of smoothness.
  \end{abstract}

\end{frontmatter}
\fancypagestyle{pprintTitle}{%
\lhead{} \chead{}\rhead{}
\lfoot{}\cfoot{}\rfoot{{\footnotesize\itshape \hfill Preprint, \today}}
\renewcommand{\headrulewidth}{0.0pt}
}

\section{Introduction}
L. Beutel et al. stated in \cite{Beutel:2002} an interesting conjecture about the
equivalence of the approximation error of the Schoenberg operator on $\ivcc{0}{1}$ and
the second order modulus of smoothness. We prove that this conjecture holds true
for the uniform Schoenberg operator if the degree of the splines is
fixed and the mesh gauge tends to zero. To this end, we
characterize the behaviour of the iterates of the Schoenberg
operator. Related to our result is the work of Zapryanova et al. \cite{Zapryanova:2012},
who proved an inverse theorem for the uniform Schoenberg
operator using the Ditzian-Totik modulus of smoothness. In contrast to
their result, we give a direct lower bound.

More specifically, we show that for $f \in \spacecf$ we have the uniform estimate
\begin{equation*}
 \omega_2(f, \delta) \leq 5 \cdot \normi{ f-\Snk f},
\end{equation*}
where $\omega_2(f,\delta)$ is the classical modulus of smoothness.

\subsection{The Schoenberg operator}
For integers $n, k > 0$, we consider the equidistant knots $\{x_j = \frac{j}{n}\}_{j=0}^{n}$ as a partition of $\ivcc{0}{1}$. 
We extend this knot sequence by setting
\begin{equation*}
  x_{-k} = \cdots = x_0 = 0 < x_1 < \ldots < x_n = \cdots = x_{n+k} = 1.
\end{equation*}
For $f \in \spacecf$, the variation-diminishing spline operator of degree $k$ with respect to
the knots $\{x_j\}_{j=-k}^{n+k}$ is then defined by
\begin{align*}
  \Snk f(x) &= \sum_{j=-k}^{n-1}f(\xi_{j,k})\bspk{j}(x),\quad 0 \leq x < 1,\\
  \Snk f(1) &= \lim_{y \nearrow 1} \Snk f(y)
\end{align*}
with the nodes
\begin{equation*}
  \xi_{j,k} := \frac{x_{j+1} + \cdots + x_{j+k}}{k},\quad -k \leq j \leq n-1,
\end{equation*}
and the normalized B-splines 
\begin{equation*}
  \bspk{j}(x) := (x_{j+k+1} - x_j)[x_j,\ldots,x_{j+k+1}](\cdot - x)_+^k.
\end{equation*}

This operator was introduced by Schoenberg in 1959 as a generalization of the Bernstein operator see, e.g., \cite{Curry:1966, Marsden:1970}.
The normalized B-splines form a partition of the unity
\begin{equation}
  \label{eq:partition_unity}
  \sum_{j=-k}^{n-1}\bspk{j}(x) = 1,
\end{equation}
and the Schoenberg operator can reproduce linear functions, i.e.,
\begin{equation}
  \label{eq:reproduce_linear}
  \sum_{j=-k}^{n-1}\xi_{j,k}\bspk{j}(x) = x,
\end{equation}
due to the chosen Greville nodes. A comprehensive overview of direct inequalities for this operator can be found in \cite{Beutel:2002}.

\subsection{Notation}
Throughout this paper, we will consider the Banach space $\spacecf$, 
i.e., the space of real-valued continuous functions on the intervall $\ivcc{0}{1}$ endowed with the supremum norm $\normi{\cdot}$,
\begin{equation*}
  \norm{f}_\infty = \sup\cset{\abs{f(x)}}{x \in \ivcc{0}{1}}, \qquad f\in \spacecf.
\end{equation*}
The space of bounded linear operators on $\spacecf$ will be denoted by $\mathcal{B}(\spacecf)$ equipped with the usual operator norm $\normop{\cdot}$. As a $n+k$-dimensional subspace of $\spacecf$, we denote by $\spacepp$ the spline space of degree $k$ with respect to the knot sequence $\setof{x_j}_{j=-k}^{n+k}$,
\begin{equation*}
  \spacepp = \cset{\sum_{j=-k}^{n-1}c_j \bspk{j}}{c_j \in \Rr,\ j \in \setof{-k,\ldots,n-1}} \subset \spacedf{k-1}.
\end{equation*}
Since $\spacepp$ is finite-dimensional, $\spacepp$ is a Banach space with the inherited norm $\normi{\cdot}$.
For more information on spline spaces see, e.g., \cite{deBoor:1987}.
For $f \in \spacecf$ and points $x_0,\ldots,x_{k} \in \ivcc{0}{1}$, the divided difference $[x_0,\ldots,x_{k}]f$ is defined to be the coefficient
of $x^k$ in the unique polynomial of degree $k$ or less that interpolates $f(x)$ at the points $x_0,\ldots,x_{k}$.

\section{The iterates of the Schoenberg operator}
In the following, we discuss some basic properties of the iterates of the Schoenberg operator. For $m \in \Nn$, 
we define 
\begin{equation*}
  (\Snkpm f)(x) = (\Snkp{m-1}(\Snk f))(x)\qquad\text{for all } x \in \ivcc{0}{1}.
\end{equation*}

\begin{lemma}
  \label{lemma:iterates}
  We can write the $m$-th iterate of the Schoenberg operator as
  \begin{align*}
    \Snkpm f(x) &= \Snkp{m-1}\left(\sum_{j=-k}^{n-1}f(\xi_{j,k})N_{j,k}(x)\right)\\
    & = \sum_{j_1,\ldots, j_m=-k}^{n-1}f(\xi_{j_1,k})N_{j_1,k}(\xi_{j_2,k})\cdots N_{j_{m-1},k}(\xi_{j_m,k}) N_{j_m,k}(x). 
  \end{align*}
\end{lemma}
\begin{proof}
  Induction over $m$.
\end{proof}

\subsection{The first and second derivative of the iterates}
In this section, we consider the derivatives and give explicit representations. 
For that, we define a discrete backward difference operator $\Delta_l$ by
\begin{equation*}
  \Delta_l f(\xi_{j,k}) := \frac{f(\xi_{j,k}) - f(\xi_{j-1,k})}{\xi_{j,l} - \xi_{j-1, l}}.
\end{equation*}
With this, we can state:

\begin{lemma}
  \label{lemma:derivative_spline}
  The following properties hold for the derivatives of the Schoenberg operator:
  \begin{align*}
    D \Snk f &= S^+_{n,k-1}\, \diffk f
    \shortintertext{and}
    D^2 \Snk f &=  S^{++}_{n,k-2}\, \diffkm \diffk f,
  \end{align*}
  where $S^+_{n,k}f$ and $S^{++}_{n,k}$ are Schoenberg operators with shifted knots defined by
  \begin{equation*}
    S^+_{n,k} f = \sum_{j=-k}^{n-1}f(\xi_{j, k+1})\bspk{j}\qquad\text{and}\qquad 
    S^{++}_{n,k} f = \sum_{j=-k}^{n-1}f(\xi_{j, k+2})\bspk{j}.
  \end{equation*}
\end{lemma}
\begin{proof}
  This lemma follows directly by the representation of the derivative, \cite{Marsden:1970}:
  \begin{align*}
    D\Snk f(x) &= \sum_{j=1-k}^{n-1}\frac{f(\xi_{j,k}) - f(\xi_{j-1,k})}{\xi_{j,k} - \xi_{j-1,k}} \bspkm{j}{1}(x),
    \intertext{and}
    D^2\Snk f(x) &= \sum_{j=2-k}^{n-1}\frac{\frac{f(\xi_{j,k}) - f(\xi_{j-1,k})}{\xi_{j,k} - \xi_{j-1,k}} -
    \frac{f(\xi_{j-1,k}) - f(\xi_{j-2,k})}{\xi_{j-1,k} - \xi_{j-2,k}}}{\xi_{j,k-1} - \xi_{j-1,k-1}} \bspkm{j}{1}(x).
  \end{align*}
  Applying the definition of the discrete backward difference operator $\Delta_l$ gives the required representation.
\end{proof}

Now we give an analogous representation for the iterates of the Schoenberg operator.
\begin{theorem}
  The first and the second derivative of the iterates of the Schoenberg operator have the following representation:
\begin{align*}
  \MoveEqLeft D \Snkpm  f(x) = \sum_{j_m = 1-k}^{n-1}
  \sum_{j_1,\ldots, j_{m-1}=-k}^{n-1} f(\xi_{j_1,k}) N_{j_1, k}(\xi_{j_2, k}) \cdot \ldots \cdot 
  N_{j_{m-2}, k}(\xi_{j_{m-1}, k}) \cdots\\ 
  &\cdots   \left[ \frac{N_{j_{m-1}, k}(\xi_{j_m, k})
    - N_{j_{m-1}, k}( \xi_{{j_m} - 1, k})}{\xi_{{j_m}, k} - \xi_{{j_m} - 1, k}}
    \right] N_{j_m,k-1}(x)\\
  &= \sum_{j_m = 1-k}^{n-1} \sum_{j_1,\ldots, j_{m-1}=-k}^{n-1} f(\xi_{j_1,k})
  N_{j_1, k}(\xi_{j_2, k}) \cdots N_{j_{m-2}, k}(\xi_{j_{m-1}, k}) \Delta_k N_{j_{m-1}}(\xi_{j_m}) N_{j_m,k-1}(x) . 
\end{align*}
and
\begin{align*}
  \MoveEqLeft D^2 \Snkpm  f(x) = \sum_{j_m = 2-k}^{n-1}
  \sum_{j_1,\ldots, j_{m-1}=-k}^{n-1} f(\xi_{j_1,k}) N_{j_1, k}(\xi_{j_2, k}) \cdot \ldots \cdot 
  N_{j_{m-2}, k}(\xi_{j_{m-1}, k}) \cdots\\ 
  &\cdots  \frac{ \left[ \frac{N_{j_{m-1}, k}(\xi_{j_m, k})
    - N_{j_{m-1}, k}( \xi_{{j_m} - 1, k})}{\xi_{{j_m}, k} - \xi_{{j_m} - 1, k}}
    \right] - 
    \left[ \frac{N_{j_{m-1}, k}(\xi_{{j_m}-1, k})
    - N_{j_{m-2}, k}( \xi_{{j_m} - 1, k})}{\xi_{{j_m}-1, k} - \xi_{{j_m} - 2, k}}
    \right]  }{\xi_{{j_m},k-1} - \xi_{{j_m}-1, k-1}} N_{j_m,k-1}(x)\\
  &= \sum_{j_m = 1-k}^{n-1} \sum_{j_1,\ldots, j_{m-1}=-k}^{n-1} f(\xi_{j_1,k})
  N_{j_1, k}(\xi_{j_2, k}) \cdots N_{j_{m-2}, k}(\xi_{j_{m-1}, k}) \Delta_{k-1} \Delta_k N_{j_{m-1}}(\xi_{j_m}) N_{j_m,k-1}(x) . 
\end{align*}
\end{theorem}
\begin{proof}
  Applying Lemma \ref{lemma:iterates} and \ref{lemma:derivative_spline} to $S_{n,k}^{m-1}f$ yields the result.
\end{proof}

\subsubsection{An upper bound for the second derivative of the iterates}
%
%
%
%
Our idea is now to work with the shift invariant basis functions $\bspk{j}$, $j \in \setof{0,\ldots,n-k-1}$, to stay 
away from the boundary of the interval $\ivcc{0}{1}$. Then we can represent the Schoenberg operator as a convolution operator and apply known techniques for this kind of operators.

Therefore, let $x \in \ivcc{x_{2k+2}}{x_{n-2k-2}}$. Then we have
\begin{equation*}
  x \not\in \bigcup_{j=-k}^{k+1} \mathrm{supp\,} \bspk{j} \text{ and } 
  x \not\in \bigcup_{j=n-2k-2}^{n-1} \mathrm{supp\,} \bspk{j},
\end{equation*}
because $\mathrm{supp\,}\bspk{j} \subset \ivcc{x_j}{x_{j+k+1}}$. Besides, we can simplify the notation of the iterates of the Schoenberg operator for $x \in \ivcc{x_{2k+2}}{x_{n-2k-2}}$ to
  \begin{equation*}
    \Snkpm f(x) = \sum_{j_1,\ldots, j_m=0}^{n-k-1}f(\xi_{j_1,k})N_{j_1,k}(\xi_{j_2,k})\cdots N_{j_{m-1},k}(\xi_{j_m,k}) N_{j_m,k}(x). 
  \end{equation*}
Now, we show that the basis functions $\setof{\bspk{j}}_{j=0}^{n-k-1}$ are shift invariant.
\begin{theorem}
  \label{thm:shift_invariance}
  The $\bspk{j}$ with $j \in \setof{0,\ldots, n-k-1}$ are translates of each other, i.e.,
  \begin{equation*}
    \bspk{j+1}(\xi_i) =  \bspk{j}(\xi_{i-1}),
  \end{equation*}
  and $\mathrm{supp\,} \mathrm{span\,}\bspk{j} \subset \ivcc{0}{1}$
\end{theorem}
\begin{proof}
  As $\mathrm{supp\,} \bspk{j} \subset \ivcc{x_{j}}{x_{j+k+1}}$ all corresponding knots $x_i$, $i \in \setof{j, \ldots, j+k+1}$ are
  distinct from each other. Explicitly, we have $x_i = \frac{i}{n}$. Now let $h = 1/n$. Then, we get
  \begin{align*}
    \bspk{j+1}(\xi_i) &= (x_{j+k+2} - x_{j+1})[x_{j+1},\ldots,x_{j+k+2}](\cdot - \xi_i)_+^k\\
                  &= (x_{j+k+1} - x_{j}) \frac1{h^k\cdot k!} \sum_{l={j+1}}^{j+k+2}\binom{k+1}{l-j-1}(-1)^{j+k+2-l}(x_l - \xi_i)_+^k\\
                  &= (x_{j+k+1} - x_{j}) \frac1{h^k\cdot k!} \sum_{l={j}}^{j+k+1}\binom{k+1}{l-j}(-1)^{j+k+1-l}(x_{l+1} - \xi_i)_+^k\\
                  &= (x_{j+k+1} - x_{j}) \frac1{h^k\cdot k!} \sum_{l={j}}^{j+k+1}\binom{k+1}{l-j}(-1)^{j+k+1-l}(x_{l} - \xi_{i-1})_+^k\\
                  &= \bspk{j}(\xi_{i-1}).
  \end{align*}
  The last line holds, because
  \begin{equation*}
    x_{l+1} - \xi_{i} = \frac{k\cdot (l+1) - \sum_{j=1}^k(i+j)}{nk} = \frac{k\cdot l - \sum_{j=1}^k(i+j-1)}{nk}\\
                     = x_{l} - \xi_{i-1}.
 \end{equation*}
\end{proof}

With Theorem \ref{thm:shift_invariance}, we get the following corollary:
\begin{corollary}
  For $m \in \Nn$ and $x \in \ivcc{x_{k+1}}{x_{n-2k-2}}$ we get:
  \begin{align*}
    \MoveEqLeft D\Snk f(x) = \sum_{j_1,\ldots, j_{m}=0}^{n-k-1} f(\xi_{j_1,k})
    N_{j_1, k}(\xi_{j_2, k}) \cdots N_{j_{m-2}, k}(\xi_{j_{m-1}, k}) \Delta_k N_{j_{m-1}}(\xi_{j_m}) N_{j_m,k-1}(x)\\
    &= \sum_{j_1,\ldots, j_{m}=0}^{n-k-1}
    f(\xi_{j_1,k})
    N_{j_1, k}(\xi_{j_2, k}) \cdots \Delta_k N_{j_{m-2}, k}(\xi_{j_{m-1}, k}) N_{j_{m-1}}(\xi_{j_m}) N_{j_m,k-1}(x)\\
    & \qquad\qquad\qquad\qquad\qquad\qquad\qquad \vdots \\
    &= \sum_{j_1,\ldots, j_{m}=0}^{n-k-1}
    f(\xi_{j_1,k}) \diffk N_{j_1, k}(\xi_{j_2, k}) N_{j_2,
      k}(\xi_{j_3, k}) \cdots N_{j_{m-1}}(\xi_{j_m}) N_{j_m,k-1}(x),
  \end{align*}
  i.e., the backward difference operator can be applied to $\bspk{j}$ for every index $j$. Thus, we have $m-1$ possibilites to represent  the first derivative of the iterated Schoenberg operator.

  Analog for $D^2 \Snkpm f$, where we have
  \begin{align*}
    \MoveEqLeft D^2\Snk f(x) = \sum_{j_1,\ldots, j_{m}=0}^{n-k-1} f(\xi_{j_1,k})
    N_{j_1, k}(\xi_{j_2, k}) \cdots N_{j_{m-2}, k}(\xi_{j_{m-1}, k}) \Delta_{k-1}\Delta_k N_{j_{m-1}}(\xi_{j_m}) N_{j_m,k-2}(x)\\
    &= \sum_{j_1,\ldots, j_{m}=0}^{n-k-1}
    f(\xi_{j_1,k}) N_{j_1, k}(\xi_{j_2, k}) \cdots \Delta_{k-1} N_{j_{m-2}, k}(\xi_{j_{m-1}, k}) \Delta_{k}N_{j_{m-1}}(\xi_{j_m}) N_{j_m,k-2}(x)\\
    & \qquad\qquad\qquad\qquad\qquad\qquad\qquad \vdots \\
    &= \sum_{j_1,\ldots, j_{m}=0}^{n-k-1}
    f(\xi_{j_1,k}) \diffkm\diffk N_{j_1, k}(\xi_{j_2, k}) N_{j_2,
      k}(\xi_{j_3, k}) \cdots N_{j_{m-1}}(\xi_{j_m}) N_{j_m,k-2}(x).
  \end{align*}
  Similar to $D\Snk f$, we have $\frac{m(m-1)}{2}$ possibilites to represent the second derivative of the $m$-th iterate of the Schoenberg operator.
\end{corollary}

We will abreviate the last term by
\begin{equation}
  \label{eq:iterates_sec_derivative}
  D^2\, \Snkpm f(x) = 
     \sum_{j_1,\ldots, j_{m}=-k}^{n-1} f(\xi_{j_1,k}) 
         \cdot P(j_1,\ldots,j_m; x) 
         \cdot I_{{l_1},{l_2}}(j_1,\ldots,j_{m-1}; x),
\end{equation}
where
\begin{equation*}
P(j_1,\ldots,j_m;x) := \left[\prod_{l=1}^{m-1}N_{j_l,k}(\xi_{j_{l+1},k})\right] N_{j_m,k-2}(x),
\end{equation*}
and for $l_1,l_2 \in \setof{1,\ldots,m-1}$, $l_1 \leq l_2$,
\begin{equation*}
 I_{{l_1},{l_2}}(j_1,\ldots,j_{m-1}; x) =
 \begin{dcases}
    \frac{\Delta_{k-1}\bspk{l_1}(x)\cdot \Delta_k \bspk{l_2}(x)}{\bspk{l_1}(x)\cdot \bspk{l_2}(x)}, & \text{for }l_1 \neq l_2,\\
   \frac{\Delta_{k-1}\Delta_k \bspk{l_1}(x)}{\bspk{l_1}(x)}, & \text{for }l_1 = l_2.
 \end{dcases}
\end{equation*}
Now we are able to give an upper bound for the second order derivative of the iterated Schoenberg operator:
\begin{theorem}
  For the integer $m \geq 2$, $h=1/n$ and $x \in \ivcc{x_{2k+2}}{x_{n-2k-2}}$ we have the upper bound
  \begin{equation*}
     \abs{D^2 \Snkpm f(x)} \leq  \frac{2\varepsilon_{n,k}}{(m-1)^{3/2} h^2} \cdot \normi{f}.
  \end{equation*}
\end{theorem}
\begin{proof}
  As we have $\frac{m(m-1)}{2}$ possibilities to express $D^2 \Snkpm f(x)$, we write \eqref{eq:iterates_sec_derivative} as the following mean:
  \begin{align*}
    D^2\, \Snkpm f(x) &= \frac2{m(m-1)} \sum_{{l_1}\leq{l_2}=1}^{m-1} D^2\, \Snkpm f(x)\\
    &= \frac{2}{m(m-1)} \sum_{j_1,\ldots, j_{m}=0}^{n-k-1}\left(
      f(\xi_{j_1,k}) \cdot P(j_1,\ldots,j_m; x) \cdot
      \sum_{{l_1}\leq{l_2}=1}^{m-1}
      I_{{l_1},{l_2}}(j_1,\ldots,j_{m-1}; x)\right).
  \end{align*}

  Since $P$ is positive, we can split $P$ into  $P=P^{1/2}P^{1/2}$, where $P^{1/2}$ is the positive root. 
Then we apply the Cauchy-Schwarz inequality and get in abbreviated notation the following pointwise inequality for $x \in \ivcc{x_{2k+2}}{x_{n-2k-2}}$:
  \begin{align}
    \abs{D^2 \Snkpm f} &\leq \frac{2}{m(m-1)}\left\{\sum_{j_1,\ldots, j_{m}=0}^{n-k-1}|f|^{2}P\right\}^{\frac1{2}}\left\{\sum_{j_1,\ldots, j_{m}=0}^{n-k-1} P\left(\sum_{{l_1}\leq{l_2}=1}^{m-1} I \right)^2\right\}^{\frac1{2}}\notag\\\label{eq:upper_bound_step1}
    &\leq  \frac{2}{m(m-1)}\left(\normi{f}\cdot 1\right)\left(\sum_{j_1,\ldots, j_{m}=0}^{n-k-1}P\cdot \left(\sum_{{l_1}\leq{l_2}=1}^{m-1} I_{{l_1},{l_2}} \right)^2\right)^{\frac1{2}}.
  \end{align}
  Here, we used the partition of unity property of the B-splines, namely that $\sum_{j=-k}^{n-1}\bspk{j}(x) = 1$ holds for all $x \in \ivcc{0}{1}$. Summation by parts, beginning with $j_1$, $j_2$, $\ldots$,  leads to
  \begin{equation*}
    \sum_{j_1,\ldots, j_{m}=-k}^{n-1}P(j_1,\ldots,j_m; x) =
    \sum_{j_m=0}^{n-k-1} N_{{j_m}, k-2}(x) \sum_{j_{m-1}=0}^{n-k-1} N_{{j_{m-1}}, k}(\xi_{j_m, k}) \cdots \sum_{{j_1}=0}^{n-k-1}
    N_{{j_1}, k}(\xi_{j_2 , k}) = 1.
  \end{equation*}
  Finally, we take the supremum norm of $f$ and obtain the inequality used for the first term.

  Next, we discuss the second product in \eqref{eq:upper_bound_step1}. For the term $\left(\sum_{{l_1}\leq{l_2}=1}^{m-1} I_{{l_1},{l_2}}
  \right)^2$ we get formally
  \begin{equation*}
    \left(\sum_{{l_1}\leq{l_2}=1}^{m-1} I_{{l_1},{l_2}} \right)^2 = 
    \sum_{l_1=l_2=1}^{m-1} I_{l_1,l_2}^2 + \sum_{l_1 \neq l_2} I_{{l_1},{l_2}} I_{{s_1},{s_2}}.
  \end{equation*}
  Note that the last sum vanishes, since for any indices $i,j \in \setof{0,\ldots,n-k-1}$ we have
  \begin{equation*}
    \sum_{j=-k}^{n-1} \diffk \bspk{j}(\xi_i) = 0,
  \end{equation*}
  and
  \begin{equation*}
    \sum_{j=-k}^{n-1} \diffkm \diffk \bspk{j}(\xi_i) = 0,
  \end{equation*}
  because of the partition of unity \eqref{eq:partition_unity}. That means, if the difference operator $\diffk$ or $\diffkm \diffk$
  is applied to $\bspk{j}$ without beeing squared, the whole sum vanishes. Therefore, we get
  \begin{equation*}
    \sum_{j_1,\ldots, j_{m}=0}^{n-k-1}P(j_1,\ldots,j_m; x)
    \left(\sum_{{l_1}\leq{l_2}=1}^{m-1} I_{{l_1},{l_2}}\right)^2 =
    \sum_{j_1,\ldots, j_{m}=0}^{n-k-1}P(j_1,\ldots,j_m; x)\sum_{l=1}^{m-1}
    I_{l,l}^2.
  \end{equation*}
  With this we obtain from \eqref{eq:upper_bound_step1} the final inequality
  \begin{align*}
    \abs{D^2 \Snkpm f(x)} &\leq \frac{2}{(m-1)^2}\normi{f} \left((m-1)\frac{\varepsilon_{n,k}^2}{h^4}\right)^{\frac1{2}}\\
                          &\leq \frac{2\varepsilon_{n,k}}{(m-1)^{3/2}h^2}\cdot \normi{f},
  \end{align*}
  where
  \begin{align*}
    \varepsilon_{n,k}^2 &:= \sup_i \sum_{j=-k}^{n-1} \frac{\left(
        \bspk{j}(\xi_{i,k}) - 2\bspk{j} (\xi_{i-1,k}) + \bspk{j}(\xi_{i-2,k})\right)^2}{\bspk{j}(\xi_{i,k}) }  
    = \sup_i \sum_{j=-k}^{n-1} \frac{(\diffkm\diffk \bspk{j}(\xi_{i,k}))^2}{\overline{N}_{j,k}(\xi_{i,k})}
    \shortintertext{with} 
    \overline{N}_{j,k}(\xi_{l,k}) &:=
    \begin{cases}
      \bspk{j}(\xi_{i,k}), & \text{if } \bspk{j}(\xi_{i,k}) = 0,\\
      1, &\text{if } \bspk{j}(\xi_{i,k}) \neq 0.\\
    \end{cases}
  \end{align*}
  The terms $\overline{N}_{j,k}(\xi_{l,k})$ are formally needed to avoid zero
  divisions in the term for $\varepsilon_{n,k}^2$.
\end{proof}

\begin{corollary}
  \label{cor:upper_bound_iterates}
  Due to the uniform convergence, we get for $k > 0$ fixed, $m > 1$ and $n \to \infty$ 
  the uniform upper bound
  \begin{equation*}
    \normi{D^2 \Snkpm f} \leq \frac{2\varepsilon_{n,k}}{h^2\cdot(m-1)^{3/2}}\cdot \normi{f}.
  \end{equation*}
\end{corollary}

\section{The lower bound of the Schoenberg operator}
In this section, we show that for $0 < t \leq \frac1{2}$ and $k \geq 3$, there exists a constant $M > 0$, such that
\begin{equation*}
  M \cdot \omega_2(f, t) \leq \normi{f - \Snk f},
\end{equation*}
where the second order modulus of smoothness $\omega_2: \spacecf \times \ivoc{0}{\frac{1}{2}} \to \ivco{0}{\infty}$ is defined by
\begin{equation*}
  \omega_2(f,t) := \sup_{0<h<t}\sup_{x \in \ivcc{0}{1-2h}}\abs{f(x) - 2f(x+h) + f(x+2h)}.
\end{equation*}

As the modulus of smoothness is equivalent to the $K$-functional \cite{Butzer:1967,Johnen:1976}, 
we can derive the inequality
\begin{equation}
  \label{eq:inequality_modulus_kfunctional}
  \omega_2(f, t) \leq  4\normi{f-\Snk f}+t^2\normi{D^2 \Snk f}.
\end{equation}
To prove our main result, we need to estimate the second term by the approximation error $\normi{f- \Snk f}$. 
In a first step, we show that the second order differential operator $D^2$ is bounded on the spline space.
\begin{lemma}
  \label{lem:bounded_derivative}
  For $k \geq 3$, the differential operator $D^2: \spacepp \to \spaceppk{k-2}$ is bounded
  with 
  \begin{equation*}
    \normop{D^2} \leq \frac{4d_k}{h^2}, 
  \end{equation*}
  where $d_k > 0$ is a constant depending only on $k$.
\end{lemma}
\begin{proof}
  Let $s \in \mathcal{S}(n, k)$, $s(x) = \sum_{j=-k}^{n-1}c_j \bspk{j}(x)$, with $\norm{s}_\infty = 1$. 
  According to M. Marsden \cite{Marsden:1970}, Lemma 2 on page 35, we can calculate the second order derivative by
  \begin{equation*}
    D^2s(x) = \sum_{j=2-k}^{n-1}\frac{ \frac{c_j - c_{j-1}}{\xi_{j,k} - \xi_{j-1,k}} - \frac{c_{j-1,k} - c_{j-2,k}}{\xi_{j-1,k} - \xi_{j-2,k}}}{\xi_{j,k-1} - \xi_{j-1,k-1}} \bspkm{j}{2}(x).
  \end{equation*}
  Then we obtain with the triangle inequality
  \begin{align*}
    \normi{D^2 s} &=  \normi{\sum_{j=2-k}^{n-1}\frac{ \frac{c_j - c_{j-1}}{\xi_{j,k} - \xi_{j-1,k}} - \frac{c_{j-1,k} - c_{j-2,k}}{\xi_{j-1,k} - \xi_{j-2,k}}}{\xi_{j,k-1} - \xi_{j-1,k-1}} \bspkm{j}{2}(x)}\\
               &\leq  \frac{\normi{c} + 2\normi{c} + \normi{c}}{h^2} \cdot \normi{ \sum_{j=1-k}^{n-1}\bspkm{j}{1}},
  \end{align*}
  where
  \begin{equation}
    \label{eq:norm_coeff}
    \normi{c} = \max\cset{\abs{c_j}}{j \in \setof{-k, \ldots, n-1}}.
  \end{equation}
  According to \cite{deBoor:1973}, there exists $d_k > 0$, such that
  \begin{equation}
    \label{eq:stability_spline}
    d_k^{-1} \norm{c}_\infty \leq \norm{\sum_{j=-k}^{n-1}c_j\bspk{j}}_\infty \leq \norm{c}_\infty.
   \end{equation}
   Rewriting the first inequality yields $\norm{c}_\infty \leq D_k$, because $\normi{s} = 1$.
   Now we use the partition of unity \eqref{eq:partition_unity} to derive the estimate
  \begin{align*}
    \normi{D^2 s} 
               &\leq \frac{4}{h^2}d_k.
  \end{align*}
  Taking the supremum of all $s \in \spacepp$ with $\normi{s} = 1$ yields the result.
\end{proof}

Now we are able to prove our main result:
\begin{theorem}
  \label{thm:lower_bound}
  For $0 < t \leq \frac1{2}$ and $k \geq 3$, there exists a constant $M > 0$ only depending on $n$ and $k$, independent of $f$, 
  such that
  \begin{equation*}
    M \cdot \omega_2(f, t) \leq \normi{f - \Snk f}.
  \end{equation*}
\end{theorem}
\begin{proof}
  We extend  $\normi{D^2 \Snk f}$ into a telescopic series:
  \begin{align*}
    \normi{D^2 \Snk f} &=  \normi{D^2 \Snk f-D^2 \Snkp{2} f+ D^2 \Snkp{2} f -D^2 \Snkp{3} f + \ldots}\\
    &\leq \sum_{m=1}^\infty \normi{D^2 \Snkpm (f-S_{n,k}f)}\\
    &= \normi{D^2 \Snk (f-S_{n,k}f)} + \sum_{m=2}^\infty \normi{D^2 \Snkpm (f-S_{n,k}f)}.
  \intertext{Then we apply Corollary \ref{cor:upper_bound_iterates} and Lemma \ref{lem:bounded_derivative} and obtain}
    \normi{D^2 \Snk f} &\leq \frac{4 d_k\normi{f-S_{n,k}f}}{h^2} + \sum_{m=1}^\infty \frac{2\varepsilon_{n,k}}{h^2\cdot m^{3/2}}\normi{f-\Snk f} \\
    &\leq \frac{4d_k + 2 \varepsilon_{n,k} \cdot \zeta(\frac{3}{2}) }{h^2}\,
    \normi{ f-\Snk f}.
  \end{align*}
  Finally, applying the above result to \eqref{eq:inequality_modulus_kfunctional} yields the estimate
  \begin{equation*}
    \omega_2(f,t) \leq \left(4 + \frac{t^2(4d_k + 2\varepsilon_{n,k}\cdot \zeta(\frac{3}{2}))}{h^2}\right)\normi{ f-\Snk f}
  \end{equation*}
\end{proof}

\begin{corollary}
  For $k \geq 3$, $n \to \infty$ and $f \in \spacecf$ the following uniform estimate holds:
  \begin{equation*}
    \omega_2(f, \delta) \leq 5 \cdot \normi{ f-\Snk f}.
  \end{equation*}
\end{corollary}
\begin{proof}
  With 
  \begin{equation*}
    \delta = \frac{h}{\sqrt{(4d_k + 2\varepsilon_{n,k}\cdot \zeta(\frac{3}{2}))}},
  \end{equation*}
  the corollary follows, because for $n \to \infty$ we have that $h \to 0$ and hence, $\delta \to 0$.
\end{proof}

\begin{corollary}
   For $0 < t \leq \frac1{2}$ and $k \geq 3$, we have the equivalence 
  \begin{equation*}
    \omega_2(f, t) \sim \normi{f - \Snk f}
  \end{equation*}
  in the sense that there exist constants $M_1, M_2 > 0$ independent of $f$ and only depending on $n$ and $k$ such that
  \begin{equation*}
    M_1 \cdot \omega_2(f, t) \leq \normi{f - \Snk f} \leq M_2 \cdot \omega_2(f, t).
  \end{equation*}
\end{corollary}
\begin{proof}
  We apply Theorem \ref{thm:lower_bound} to get the lower inequality and we use the inequality
  \begin{equation*}
    \normi{f - \Snk f} \leq 
    \left(1 + \frac1{2t^2}\cdot\min\setof{\frac{1}{2k},\, \frac{(k+1)H^2}{12}}\right)\cdot \omega_2(f, t),
  \end{equation*}
  from \cite{Beutel:2002} 
  to obtain the upper bound, where
  \begin{equation*}
    H := \max\cset{(x_{j+1} - x_j)}{j \in \setof{-k,\ldots,n-1}}.
  \end{equation*}  
\end{proof}

Consequently, we have proved that the conjecture stated in \cite{Beutel:2002} holds true under the conditions of
Theorem~\ref{thm:lower_bound}. Additionally, we note that in Corollary~3 we have the relation $d_k \sim 2^k$. Therefore, 
$\delta$ tends to zero also for $k \to \infty$. With this note, we finally conclude with the following related conjecture:
\begin{conjecture}
  For $n > 0$ fixed and $k \to \infty$, there exists $M > 0$ independent on $n$ and $k$ such that
  \begin{equation*}
    M \cdot \omega_2(f, \delta) \leq \normi{ f-\Snk f}.
  \end{equation*}
\end{conjecture}


\bibliographystyle{elsarticle-num}
\bibliography{references}

\end{document}